\documentclass{article}

\usepackage{amsfonts,amsthm,amssymb}
\usepackage{bbm}
\usepackage{hyperref}
\usepackage{mathtools}
\usepackage{multicol}
\usepackage{tikz}
\usepackage{pgfplots}
\usepackage{xcolor}
\usepackage{xurl}
\usetikzlibrary{decorations.markings,angles,calc,quotes}
\pgfplotsset{compat=1.18}

\newcommand{\bbC}{\mathbb{C}}
\newcommand{\bbN}{\mathbb{N}}

\newcommand{\bbR}{\mathbb{R}}
\newcommand{\bbZ}{\mathbb{Z}}

\newcommand{\calB}{\mathcal{B}}
\newcommand{\calF}{\mathcal{F}}
\newcommand{\calG}{\mathcal{G}}

\newcommand{\calO}{\mathcal{O}}

\newcommand{\rme}{\mathrm{e}}
\newcommand{\rmi}{\mathrm{i}}
\newcommand{\dd}{\,\mathrm{d}}

\DeclarePairedDelimiter{\norm}{\lVert}{\rVert}
\DeclarePairedDelimiter{\abs}{\lvert}{\rvert}
\DeclarePairedDelimiterX{\set}[2]\lbrace\rbrace{%
  #1 \,\delimsize\vert\, #2 
}
\renewcommand{\Re}{\operatorname{Re}}
\renewcommand{\Im}{\operatorname{Im}}

\newtheorem{theorem}{Theorem}
\newtheorem{proposition}{Proposition}
\newtheorem{corollary}{Corollary}

\newtheorem*{problem*}{Problem}

\theoremstyle{remark}
\newtheorem{remark}{Remark}
\newtheorem{example}{Example}

\title{On a dense set of functions determined by sampled Gabor magnitude}
\author{Matthias Wellershoff\textsuperscript{a}}
\date{\today}

\begin{document}

    \maketitle

    \vspace{-15pt}
    \begin{center}
        \textsuperscript{a} Department of Mathematics, University of Maryland, \\
        4176 Campus Drive -- William E.~Kirwan Hall, College Park, MD 20742-4015 \\
        Email: \href{mailto:wellersm.math@gmail.com}{wellersm.math@gmail.com}
    \end{center}
    \vspace{10pt}

    \begin{abstract}
        We study the problem of recovering a function from the magnitude of its Gabor transform sampled on a discrete set. While it is known that uniqueness fails for general square integrable functions, we show that phase retrieval is possible for a dense class of signals: specifically, those whose Bargmann transforms are entire functions of exponential type. Our main result characterises when such functions can be uniquely recovered (up to a global phase) from magnitude only data sampled on uniformly discrete sets of sufficient lower Beurling density. In particular, we prove that every entire function of exponential type is uniquely determined (up to a global phase) among all second order entire functions by its modulus on a sufficiently dense shifted lattice with suitable structure.
    \end{abstract}

    \paragraph{Keywords} Phase retrieval, Gabor transform, Bargmann transform, entire functions, exponential type, Beurling density, sampling

    \section{Introduction}

    We consider the problem of recovering an entire function $f$ of second order from its modulus $\abs{f}$ known only on a discrete set $\Lambda \subset \bbC$. This is known as the \emph{phase retrieval problem for entire functions} and has been investigated in \cite{Liehr2024Arithmetic,McDonald2004Phase,Wellershoff2024Phase}. 

    Our interest in this problem comes from \emph{Gabor phase retrieval}, where one seeks to recover a square integrable function $f \in L^2(\bbR)$ from the modulus of its Gabor transform,
    \begin{equation}\label{eq:def_Gabor}
        \calG f (x,\omega) := 2^{1/4} \int_\bbR f(t) \rme^{-\pi(t-x)^2 - 2 \pi \rmi t \omega} \dd t, \qquad (x,\omega) \in \bbR^2.
    \end{equation}
    Through a unitary mapping (the Bargmann transform) from $L^2(\bbR)$ onto a space of entire functions (the Fock space), this problem becomes equivalent to recovering an entire function from its modulus. We elaborate on this connection in Section~\ref{sec:FockSpace}.
    
    A fundamental feature of any phase retrieval problem is that the function $f$ and any scalar multiple $\tau f$, with $\tau \in C_1 = \set{z \in \bbC}{\abs{z} = 1}$, yield the same measurements $\abs{f}$ (or $\abs{\calG f}$). Thus, it is impossible to distinguish them from magnitude only data. Therefore, recovery is typically understood \emph{up to a global phase}; that is, with respect to the equivalence relation
    \begin{equation*}
        f \sim g :\iff \exists \tau \in C_1 : f = \tau g.
    \end{equation*}

    The phase retrieval problem arises in different areas of physics whenever detectors record intensity without phase information. Examples include X-ray crystallography, astronomy, electron microscopy or quantum tomography \cite{Allain2025Phasebook}. Many audio processing techniques also rely on phase retrieval in one form or another since magnitude spectrograms are often used as natural representations of audio signals and phase must be recovered or estimated to reconstruct the signals \cite{Flanagan1966Phase,Griffin1984Signal,Wang2018End}. The phase retrieval problem is highly ill-posed and fundamental questions about the uniqueness and stability of solutions as well as the design of efficient algorithms to find them remain open. In spite of this, the field has witnessed significant advances, driven by theoretical progress and rapid developments across imaging and signal processing \cite{Allain2025Phasebook}. A considerable gap, however, remains between the theoretical understanding of phase retrieval methods and their practical implementation in real-world scenarios. The present paper aims to narrow this gap.

    The central question that we want to address is whether Gabor phase retrieval is generically possible or impossible: are there more signals that can be recovered or more that cannot? What is the exact structure of the recoverable set? Until now, the Gaussian is essentially the only signal known to be recoverable \cite{Alaifari2024On}. The present paper significantly extends this picture, offering partial answers to these questions.

    \subsection{Overview of the literature on Gabor phase retrieval}
    
    The Gabor phase retrieval problem has attracted a lot of attention in recent years starting with the development of uniqueness results for compactly supported functions \cite{Alaifari2021Uniqueness,Grohs2023Injectivity,Grohs2025From,Iwen2023Phase,Wellershoff2023Sampling,Wellershoff2024Injectivity}. These works showed that, under suitable sampling conditions, a compactly supported function can be uniquely recovered (up to a global phase) from the modulus of its Gabor transform.

    However, when extending the problem to the full space $L^2(\bbR)$, uniqueness fails: it has been shown \cite{Alaifari2022Phase} that, for any set $\Lambda$ consisting of infinitely many equidistant parallel lines in the time-frequency plane $\bbR^2$, there exist functions $f, g \in L^2(\bbR)$ such that $f \not\sim g$ yet $\abs{\calG f} = \abs{\calG g}$ on $\Lambda$.\footnote{In fact, there exist countably many such sets $(\Lambda_n)_{n=1}^\infty \subset \bbR^2$, each consisting of infinitely many parallel lines, and a single function $f \in L^2(\bbR)$ which cannot be recovered from the entire collection of measurements $(\abs{\calG f}|_{\Lambda_n})_{n=1}^\infty$ \cite{Wellershoff2024Phase}.} Moreover, \cite{Alaifari2024On} showed that the set of such \emph{counterexamples},
    \begin{equation*}
        \mathfrak{C}(\Lambda) := \left\{ f \in L^2(\bbR) \,\middle|\, \exists g \in L^2(\bbR): f \not\sim g \text{ and } \abs{\calG f}|_\Lambda = \abs{\calG g}|_\Lambda \right\},
    \end{equation*}
    is dense in $L^2(\bbR)$ regardless of the choice of infinitely many equidistant parallel lines $\Lambda$. At the same time, the Gaussian $\smash{\phi(t) = 2^{1/4} \rme^{-\pi t^2}}$ can be recovered from samples of $\abs{\calG \phi}$ on a sufficiently dense square lattice $a \bbZ^2$, namely, for $a < 1$.

    These results provide the background for the following questions, formulated in \cite{Allain2025Phasebook}: define the set of \emph{examples}
    \begin{equation*}
        \mathfrak{E}(\Lambda) := \left\{ f \in L^2(\bbR) \,\middle|\, \forall g \in L^2(\bbR): \abs{\calG f}|_\Lambda = \abs{\calG g}|_\Lambda \Rightarrow f \sim g \right\}.
    \end{equation*}
    Is $\mathfrak{E}(\Lambda)$ dense in $L^2(\bbR)$? And if so, which of the two sets, $\mathfrak{E}(\Lambda)$ or $\mathfrak{C}(\Lambda)$, is ``larger'' in an appropriate sense?

    In this paper, we answer the first of these questions affirmatively: we show that $\mathfrak{E}(\Lambda)$ is dense in $L^2(\bbR)$ provided that $\Lambda \subset \bbR^2$ is a uniformly discrete set (i.e., the points are separated by a minimal distance) with sufficiently large lower Beurling density (see Section~\ref{sec:Notation}) and suitable structure. Our approach builds on the ideas developed in \cite{Alaifari2024On}.

    \begin{remark}[Why results for compactly supported functions do \emph{not} settle the question]
        At first glance, one might expect the density of $\mathfrak{E}(\Lambda)$ in $L^2(\bbR)$ to follow from uniqueness results for compactly supported functions since all such functions are examples. However, this is not the case: the required sampling density in frequency generally depends on the size of the time support. For instance, \cite{Wellershoff2023Sampling} shows that all $f \in L^2([-A,A])$ can be recovered (up to a global phase) from $\abs{\calG f}$ sampled on $\lbrace 0, 1 \rbrace \times \bbZ/(4A)$; note that the sampling density grows as $A$ increases.
    \end{remark}

    \subsection{Main results}

    Our main result establishes that entire functions of exponential type can be uniquely distinguished (up to a global phase) from all other second order entire functions based solely on knowledge of their modulus on uniformly discrete sets $\Lambda$ with sufficient lower Beurling density and suitable geometric structure. We present it here in simplified form. For the terminology used, please consult Section~\ref{sec:Notation}.

    \begin{theorem}[Main theorem; cf.~Theorem~\ref{thm:MainTheorem}]\label{thm:MainThmIntro}
        Let $f \in \calO(\bbC)$ have order at most two and type $\tau < \infty$, with $\tau = 0$ if the order is less than two, and let $g \in \calO(\bbC)$ be an entire function of exponential type at most $\kappa < \infty$. Let $\Lambda \subset \bbC$ be a uniformly discrete set with lower Beurling density $D^{-}(\Lambda) > 2\tau/\pi$. Suppose that $\Lambda$ contains two sequences lying on parallel rays with $2 \kappa$ times their distance strictly less than $\pi$; each satisfying a density condition (see \eqref{eq:FuchsCondition}) after an appropriate translation-rotation. Then,
        \begin{equation*}
            f \sim g \iff \abs{f} = \abs{g} \text{ on } \Lambda.
        \end{equation*}
    \end{theorem}

    In particular, the result applies when $\Lambda$ is a so called \emph{shifted lattice}, i.e.,
    \begin{equation*}
        \Lambda = z_0 + \omega_1 \bbZ + \omega_2 \bbZ \subset \mathbb{C},
    \end{equation*}
    where $z_0, \omega_1, \omega_2 \in \bbC$ and $\omega_1/\omega_2 \not\in \bbR$, provided that the lattice is sufficiently dense (i.e., its covolume is strictly less than $\pi/(2\tau)$) and that the geometric conditions
    \begin{equation*}
        2 \kappa \cdot \abs{\omega_2} < \pi, \qquad 2 \kappa \cdot \frac{\abs{\Im(\omega_1 \overline{\omega_2})}}{\abs{\omega_2}} < \pi
    \end{equation*}
    are satisfied (cf.~Corollary~\ref{cor:ShiftedLattice}). In view of the counterexamples from \cite{Alaifari2022Phase}, these geometric constraints are sharp for rectangular shifted lattices, namely, for $\Re(\omega_1 \overline \omega_2) = 0$ (cf.~Section~\ref{sec:Counterexamples}).

    \begin{example}[Polynomials; cf.~Section~\ref{sec:Polynomials}]
        Let $g \in \bbC[z]$ be a polynomial. Then, our main theorem (in stronger form) implies that, for any entire function $f \in \calO(\bbC)$ of order at most two and type $\tau$, with $\tau = 0$ if the order is less than two, and any uniformly discrete set $\Lambda \subset \bbC$ with lower Beurling density $D^{-}(\Lambda) > 2\tau/\pi$, we have
        \begin{equation*}
            f \sim g \iff \abs{f} = \abs{g} \text{ on } \Lambda.
        \end{equation*}
        Using the Bargmann transform (which will be introduced in the next section), this yields a result for Gabor phase retrieval: let $f \in L^2(\bbR)$ and let $g \in L^2(\bbR)$ be a finite linear combination of Hermite functions (the preimages of monomials under the Bargmann transform). Suppose $\Lambda \subset \bbR^2$ is a uniformly discrete set with lower Beurling density $D^{-}(\Lambda) > 1$. Then,
        \begin{equation*}
            f \sim g \iff \abs{\calG f} = \abs{\calG g} \text{ on } \Lambda,
        \end{equation*}
        which implies that $\mathfrak{E}(\Lambda)$ is dense in $L^2(\bbR)$ (as explained in the following section).
    \end{example}

    Finally, we note that the proof of our main result relies on a striking property of second order entire functions: namely, their growth can be determined from their behaviour on uniformly discrete sets with sufficiently large lower Beurling density. This extends an earlier result of Earl \cite{Earl1966On}.

    \begin{remark}[A simple generalisation of the main theorem and its implications for Gabor phase retrieval]\label{rem:review3}
        As noted by one of the reviewers, Theorem~\ref{thm:MainThmIntro} is \emph{not} invariant under the action of $\mathbb R^2$ on the Fock space via time-frequency translation: more precisely, consider the action of $\mathbb R^2$ on $L^2(\mathbb R)$ by
        \begin{equation*}
            \pi(x,\omega) f(t) := \mathrm e^{\pi \mathrm i x \omega} \mathrm e^{-2 \pi \mathrm i t \omega} f(t-x), \qquad t \in \mathbb R.
        \end{equation*}
        A direct calculation shows that, under the Bargmann transform, this action is equivalent to 
        \begin{equation*}
            \Pi(w) f(z) := \mathrm e^{\pi z \overline w - \frac{\pi}{2} \lvert w \rvert^2} f(z-w), \qquad z \in \mathbb C,
        \end{equation*}
        i.e., $\Pi(x + \mathrm i \omega) \circ \mathcal B = \mathcal B \circ \pi(x,\omega)$. The entire functions of exponential type $\kappa$ are \emph{not} invariant under $\Pi$: indeed, the constant $1 \in \mathcal F^2(\mathbb C)$ has exponential type zero but $\Pi(w) 1$ has exponential type $\lvert w \rvert$.
        
        Notably, however, all other assumptions in Theorem~\ref{thm:MainThmIntro} are invariant under $\Pi$: (i) if $f$ is of order at most two and type $\tau$, then $\Pi(w)f$ is too; and (ii) if $\Lambda$ satisfies the assumptions of the theorem, then so does $\Lambda - w$. Therefore, Theorem~\ref{thm:MainThmIntro} has a straight-forward generalisation to entire functions of exponential type for which there exists $w \in \mathbb C$ such that $\Pi(w) g$ has type at most $\kappa < \infty$. Such functions $g$ can be described more elegantly as well if one is willing to introduce more concepts from complex analysis. Specifically, consider the \emph{indicator function} of an entire function of exponential type $g$,
        \begin{equation*}
            h_g(\theta) := \limsup_{r \to \infty} \frac{\log \lvert g(r\mathrm e^{\mathrm i \theta}) \rvert}{r}, \qquad \theta \in [0,2\pi],
        \end{equation*}
        which measures the growth of $g$ along the ray $\set{z \in \mathbb C}{ \arg(z) = \theta }$. Following \cite[Section~9.1 on pp.~63--65]{Levin1996Lectures}, the indicator of $g$ is the \emph{supporting function} of a convex compact set $I_g \subset \mathbb C$ called the \emph{indicator diagram} of $g$; i.e., 
        \begin{equation*}
            h_g(\theta) = \sup_{z \in I_g} \Re(z \mathrm e^{-\mathrm i \theta}).
        \end{equation*}
        $I_g$ gives a geometrical representation of the growth of $g$ in various directions. One can readily see that $g$ is of exponential type at most $\kappa$ if and only if $I_g \subset \overline B_\kappa$. Additionally, one can show that the action of $\Pi$ on the Fock space shifts the indicator diagram:
        \begin{equation*}
            I_{\Pi(w) g} = I_g + \lbrace \pi w \rbrace, \qquad w \in \mathbb C. 
        \end{equation*}
        Therefore, the entire functions of exponential type $g$ for which there exists $w \in \mathbb C$ such that $\Pi(w) g$ has type at most $\kappa$ are exactly the \emph{entire functions of exponential type whose indicator diagram is contained in some closed ball of radius $\kappa$}. Theorem~\ref{thm:MainThmIntro} continues to hold for precisely those functions, and is invariant under time-frequency translations in this formulation.

        For the Gabor transform, the following result follows immediately: let $f,g \in L^2(\mathbb R)$ and suppose that there exist $\kappa > 0$ and $(x_0,\omega_0) \in \mathbb R^2$ such that
        \begin{equation*}
            \lvert \mathcal G g (x,\omega) \rvert \lesssim \mathrm e^{\kappa \lvert (x,\omega) - (x_0, \omega_0) \rvert - \frac{\pi}{2}\lvert (x,\omega) - (x_0,\omega_0)\rvert^2}, \qquad (x,\omega) \in \mathbb R^2.
        \end{equation*}
        Let $\Lambda \subset \bbC$ be a uniformly discrete set with lower Beurling density $D^{-}(\Lambda) > 1$, and suppose that $\Lambda$ contains two sequences lying on parallel rays with $2 \kappa$ times their distance strictly less than $\pi$; each satisfying a density condition (see \eqref{eq:FuchsCondition}) after an appropriate translation-rotation. Then,
        \begin{equation*}
            f \sim g \iff \abs{\calG f} = \abs{\calG g} \text{ on } \Lambda.
        \end{equation*}

        We finally note that the space of polynomials is \emph{not} closed under $\Pi$ either, and that our results for polynomials in Section~\ref{sec:Polynomials} admit an immediate generalisation to the space of entire functions $g$ of exponential type for which $\Pi(w)g$ is a polynomial for some $w \in \mathbb C$ too. As before, this space can be described more elegantly. Specifically, it coincides exactly with the space of \emph{exponential polynomials}; i.e., functions of the form $\mathrm e^{c z} p(z)$, where $c \in \mathbb C$ and $p \in \mathbb C[z]$. Moreover, it coincides exactly with the space of entire functions of exponential type with finitely many zeroes as can be seen from Hadamard's factorisation theorem. 
    \end{remark}

    \subsection{Some preliminaries on the Fock space}\label{sec:FockSpace}

    The Gabor transform, as defined in equation~\eqref{eq:def_Gabor}, provides information about the frequencies $\omega$ present at specific times $x$ in a signal $f$. A closely related transform is the \emph{Bargmann transform}, defined by
    \begin{equation*}
        \calB f(z) := 2^{1/4} \int_\bbR f(t) \rme^{2 \pi t z - \pi t^2 - \tfrac{\pi}{2} z^2} \dd t, \qquad z \in \bbC.
    \end{equation*}
    The precise relationship between the Gabor and Bargmann transforms is given by \cite[Proposition~3.4.1(a) on p.~54]{Groechenig2001Foundations}:
    \begin{equation}\label{eq:GaborBargmannRelation}
        \calG f(x, -\omega) = \rme^{\pi \rmi x \omega} \cdot \calB f(x + \rmi \omega) \cdot \rme^{-\tfrac{\pi}{2}(x^2 + \omega^2)}, \qquad x, \omega \in \bbR.
    \end{equation}
    This identity is immensely useful because the Bargmann transform defines a unitary operator from $L^2(\bbR)$ onto the \emph{Fock space} $\calF^2(\bbC)$ of entire functions $F : \bbC \to \bbC$ for which the norm
    \begin{equation*}
        \norm{F}_\calF := \left( \int_\bbC \abs{F(z)}^2 \rme^{-\pi \abs{z}^2} \dd z \right)^{1/2}
    \end{equation*}
    is finite \cite[Theorem~3.4.3 on p.~56]{Groechenig2001Foundations}.

    As a consequence, the problem of recovering a function $f \in L^2(\bbR)$ from the magnitude of its Gabor transform $\abs{\calG f}$ sampled on a discrete set $\Lambda \subset \bbR^2$ is equivalent to recovering the entire function $\calB f$ from its modulus $\abs{\calB f}$ sampled on the reflected set $\smash{\overline \Lambda := \set{ \overline \lambda  }{ \lambda \in \Lambda }}$.

    The Bargmann transform maps every $f \in L^2(\bbR)$ to an entire function of order at most two and type at most $\pi/2$. More precisely, \cite[Proposition~3.4.1(b) and Theorem~3.4.2(b) on p.~54]{Groechenig2001Foundations}:
    \begin{equation}\label{eq:bound_type_Bargmann}
        \abs{\calB f(z)} \leq \norm{f}_{L^2} \cdot \rme^{\tfrac{\pi}{2} \abs{z}^2}, \qquad z \in \bbC.
    \end{equation}
    Finally, the Fock space $\calF^2(\bbC)$ has an orthonormal basis consisting of monomials \cite[Theorem~3.4.2(a) on p.~54]{Groechenig2001Foundations}:
    \begin{equation*}
        e_n(z) := \left( \frac{\pi^n}{n!} \right)^{1/2} z^n, \qquad z \in \bbC,~n \in \bbN_0.
    \end{equation*}
    Since $\calB$ is unitary, the preimages of these monomials under $\calB$ form an orthonormal basis for $L^2(\bbR)$. We denote them by $h_n := \calB^{-1} e_n$ and refer to them as the \emph{Hermite functions} as is usual.

    \subsection{Auxiliary results on holomorphic functions}
    
    In the proof of our main result, we will need two classical results about holomorphic functions. For the convenience of the reader, we summarise them here. 

    The first is the following generalisation of Carlson's theorem due to Fuchs \cite{Fuchs1946Generalization}, which gives a uniqueness criterion for functions of exponential type based on the density of their zeros along a single ray. 

    \begin{theorem}[{Fuchs' theorem; cf.~\cite[Section~9.5 on pp.~157--158]{Boas1954Entire}}]\label{thm:Fuchs}
        Let $f$ be holomorphic on the right half-plane $\mathbb H := \set{ z \in \mathbb C }{ \Re z > 0 }$, continuous on its closure, and suppose that there exists a constant $\kappa \in [0,\infty)$ such that $|f(z)| \lesssim \mathrm{e}^{2 \kappa |z|}$ for $z \in \mathbb H$.
        Let $(\lambda_n)_{n\in\mathbb N}$ be a uniformly discrete sequence of positive real numbers satisfying
        \begin{equation}\label{eq:FuchsCondition}
            \limsup_{r \to \infty} r^{-2\kappa/\pi} \exp\left( \sum_{\lambda_n < r} \frac{1}{\lambda_n} \right) = \infty.
        \end{equation}
        If $f(\lambda_n)=0$ for all $n\in\mathbb N$, then $f = 0$ identically.
    \end{theorem}

    \begin{remark}
    \label{rem:FuchsSharp}
        \begin{enumerate}
            \item The density condition~\eqref{eq:FuchsCondition} is sharp: if it is violated, then there exists a nontrivial function satisfying all the hypotheses of Theorem~\ref{thm:Fuchs} and vanishing on $(\lambda_n)_{n\in\mathbb N}$.

            \item The conclusion of the theorem remains valid under the stronger density condition
            \[
                \liminf_{n \to \infty} \frac{n}{\lambda_n} > \frac{2\kappa}{\pi},
            \]
            since this condition implies~\eqref{eq:FuchsCondition}.
        \end{enumerate}
    \end{remark}

    The second is Hadamard's factorisation theorem, which describes how to write entire functions of finite order in terms of their zeroes, similarly to how polynomials factor over their roots.

    \begin{theorem}[{Hadamard's factorisation theorem; cf.~e.g.~\cite[Section~8.24 on p.~250]{Titchmarsh1939Theory}}]\label{thm:HadmardFactorisation}
        Let $f \in \calO(\bbC)$ be a nonzero entire function of finite order $\rho \geq 0$ with a zero of order $k \in \bbN_0$ at the origin and nonzero zeroes $(z_j)_{j=1}^\infty \in \bbC \setminus \{0\}$ listed with multiplicity. Then, $f$ factors as 
        \begin{equation*}
            f(z) = z^k \rme^{P(z)} \prod_{j = 1}^\infty E\left( \frac{z}{z_j}; p \right),
        \end{equation*}
        where $P \in \bbC[z]$ is a polynomial of degree $q \leq \rho$, $E(u; p) := (1 - u)\exp(u + u^2/2 + \dots + u^p/p)$ is the Weierstrass primary factor of genus $p \in \bbN_0$, and $p \leq \rho$ is the smallest natural number such that $\sum_{j=1}^\infty \abs{z_j}^{-p-1} < \infty$.
    \end{theorem}

    \subsection{Notation and terminology}\label{sec:Notation}
    
    We use the convention $\bbN = \lbrace 1, 2, \dots \rbrace$ for the natural numbers and denote $\bbN_0 = \bbN \cup \lbrace 0 \rbrace$. For a finite set $S$, we denote its cardinality by $\abs{S}$. The open, complex ball of radius $r > 0$ around $z_0 \in \bbC$ is denoted by $B_r(z_0) := \set{z \in \bbC}{ \abs{z-z_0} < r }$, and we use the shorthand $B_r := B_r(0)$. The complex circle of radius $r > 0$ around the origin is denoted by $C_r := \set{ z \in \bbC }{ \abs{z} = r }$.
    
    A discrete set of complex numbers $\Lambda = \lbrace \lambda_j \rbrace_{j = 1}^\infty \subset \bbC$ is \emph{uniformly discrete} if
    \begin{equation*}
        \delta(\Lambda) := \inf_{\substack{j,k \in \bbN\\j \neq k}} \abs{\lambda_j - \lambda_k} > 0.
    \end{equation*}
    Let $K \subset \mathbb{C}$ be a compact set of unit measure. Then, the \emph{lower Beurling density} (or \emph{lower uniform density}) of $\Lambda$ is
    \begin{equation*}
        D^{-}(\Lambda) := \liminf_{r \to \infty} \frac{ \inf_{z_0 \in \bbC}  \abs{\Lambda \cap (rK + z_0) } }{ r^2 }.
    \end{equation*}
    This definition is independent of the particular choice of $K$.
    
    The polynomials with complex coefficients (and complex argument) are denoted by $\bbC[z]$. The space of entire functions is denoted by $\calO(\bbC)$. We say that an entire function $f$ is of \emph{order} $\rho \in [0,\infty]$ if 
    \begin{equation*}
        \limsup_{r \to \infty} \frac{\log \log \max_{z \in C_r} \abs{f(z)}}{\log r} = \rho.
    \end{equation*}
    Functions of order $\rho \leq 1$ are said to be of \emph{exponential type}. If $\rho$ is finite and non-zero, we say that the entire function $f$ is of \emph{type} $\tau \in [0,\infty]$ if 
    \begin{equation*}
        \limsup_{r \to \infty} \frac{\log \max_{z \in C_r} \abs{f(z)}}{r^\rho} = \tau.
    \end{equation*}

    Finally, if $f(x)$, $g(x)$ are two families of objects parametrised by $x \in X$, where $X$ is some set, then we write $f \lesssim g$ if there exists a constant $c > 0$ such that, for all $x \in X$, $f(x) \leq c g(x)$. Similarly, we write $f \gtrsim g$ if $g \lesssim f$ as well as $f \eqsim g$ if $f \lesssim g$ and $f \gtrsim g$ are true at the same time. We use $\lesssim_Q$, $\gtrsim_Q$ and $\eqsim_Q$ if the aforementioned constant $c$ only depends on some quantity $Q$.

    \section{Main results}

    Throughout this section, we consider entire functions $f \in \calO(\bbC)$ for which the following growth condition holds:
    \begin{equation}
        \liminf_{r \to \infty} \frac{\log \max_{z \in C_r} \abs{f(z)}}{r^2} < \infty. \label{eq:focklikecondition}
    \end{equation}
    We denote the left-hand side of the above inequality by $\underline \tau = \underline \tau(f)$. Condition~\eqref{eq:focklikecondition} is slightly weaker than requiring that $f$ is of order at most two and type at most $\tau \in [0,\infty)$ but suffices for all our main results. We highlight it here since it will be used repeatedly in what follows.

    \subsection{Determining the growth of entire functions from uniformly discrete samples}

    We begin with a cornerstone result of this paper: a modification of \cite[Theorem~2 on p.~724]{Earl1966On}, which itself is a generalisation of earlier work in \cite{Cartwright1938On,Pfluger1937On}. Our version shows that the growth of entire functions of second order can be determined from their behaviour on uniformly discrete sets with sufficiently large lower Beurling density. 

    The proof draws inspiration from both the original result as well as ideas in sampling theory \cite{Beurling1989Harmonic,Seip1992DensityI,Seip1992DensityII}.

    \begin{theorem}\label{thm:Earl}
        Let $f \in \calO(\bbC)$ be an entire function satisfying condition~\eqref{eq:focklikecondition}. Let $\Lambda = \lbrace \lambda_j \rbrace_{j = 1}^\infty \subset \bbC$ be a uniformly discrete set with lower Beurling density $D^{-}(\Lambda) > 2 \underline \tau / \pi$ such that 
        \begin{equation*}
            \kappa := \limsup_{j \to \infty} \frac{\log \abs{f(\lambda_j)}}{H(\abs{\lambda_j})} < \infty,
        \end{equation*}
        where $H : (0,\infty) \to (0,\infty)$ is non-decreasing and unbounded at infinity such that $r^{-2} H(r)$ is non-increasing and tends to zero as $r \to \infty$. Then, 
        \begin{equation*}
            \limsup_{r \to \infty} \frac{\log \max_{z \in C_r} \abs{f(z)}}{H(r)} \leq \kappa.
        \end{equation*}
    \end{theorem}

    \begin{remark}
        As noted in \cite{Earl1966On}, the statement of the above theorem is nicely completed by \cite[Satz~4 on p.~202--203]{Pfluger1945Ueber}, which shows that the result remains true when $H$ is bounded, as well as \cite[Corollary~3 on p.~716]{Noble1951Nonmeasurable}, which shows that the result remains true when $r^{-2} H(r)$ tends to $c$ such that $\kappa \cdot c \in (0,\underline \tau)$.
    \end{remark}

    \begin{proof}[Proof of Theorem~\ref{thm:Earl}]
        Assume, without loss of generality, that $\underline \tau < \pi/2$ and $D^{-}(\Lambda) > 1$ by applying the linear transform 
        \begin{equation*}
            z \mapsto (D^{-}(\Lambda)/2 + \underline \tau/\pi)^{1/2} z
        \end{equation*}
        if necessary. In this proof, we will make use of the well-known notion of two uniformly discrete sets $\Xi = \lbrace \xi_{m,n} \rbrace_{m,n \in \bbZ}$ and $\Xi' = \lbrace \xi'_{m,n} \rbrace_{m,n \in \bbZ}$ being \emph{uniformly close} if 
        \begin{equation*}
            \Delta_2(\Xi,\Xi') := \sup_{m,n \in \bbZ} \abs{\xi_{m,n} - \xi'_{m,n}} < \infty;
        \end{equation*}
        or, equivalently, 
        \begin{equation*}
            \Delta_\infty(\Xi,\Xi') := \sup_{m,n \in \bbZ} \max\lbrace \abs{\Re \xi_{m,n} - \Re \xi'_{m,n}}, \abs{\Im \xi_{m,n} - \Im \xi'_{m,n}} \rbrace < \infty.
        \end{equation*}
        Due to a standard argument by Beurling, we can assume that $\Lambda$ is uniformly close to the square lattice $\bbZ + \rmi \bbZ$: indeed, since $D^{-}(\Lambda) > 1$, for all $r > 0$ sufficiently large and all $z_0 \in \bbC$, 
        \begin{equation*}
            \abs{\Lambda \cap (r K + z_0)} > r^2,
        \end{equation*}
        where $K = [-1/2,1/2] + \rmi [-1/2,1/2]$ is the unit square centered at the origin. So, pick the smallest $k \in \bbN$ that is sufficiently large in the above sense and tile the complex plane by squares with side length $k$ which contain exactly $k^2$ elements of $\Lambda$. We denote those squares by $\lbrace S_\ell \rbrace_{\ell=1}^\infty$ and note that $\abs{\Lambda \cap S_\ell} > k^2$ for $\ell \in \bbN$. We can, therefore, pick exactly $k^2$ elements from $\Lambda \cap S_\ell$, for each $\ell \in \bbN$, to create a subset $\Lambda_0 \subset \Lambda$. Now, in each $S_\ell$, $\Lambda_0$ and $\bbZ + \rmi \bbZ$ have the same number of elements and, so, we may index $\Lambda_0 = \lbrace \lambda_{m,n} \rbrace$ in such a way that $\lambda_{m,n} \in S_\ell \iff m+\rmi n \in S_\ell$. Therefore, $\Lambda_0$ and $\bbZ + \rmi \bbZ$ are uniformly close because 
        \begin{equation*}
            \Delta_\infty(\Lambda_0,\bbZ + \rmi \bbZ) \leq k < \infty.
        \end{equation*}

        Next, let $M,N \in \bbZ$ be arbitrary but fixed and introduce the entire function
        \begin{equation*}
            g_{M,N}(z) := (z - \lambda_{M,N}) \rme^{ -\alpha_{M,N} (z - \lambda_{M,N})^2 / 2 } \prod_{\substack{m,n \in \bbZ\\(m,n) \neq (M,N)}} E\left(\frac{z - \lambda_{M,N}}{\lambda_{m,n} - \lambda_{M,N}},2\right),
        \end{equation*}
        where 
        \begin{equation*}
            \alpha_{M,N} := \sum_{\substack{m,n \in \bbZ\\(m,n)\neq(M,N)}} \frac{1}{(\lambda_{m,n} - \lambda_{M,N})^2}.
        \end{equation*}
        Denote $\delta := \delta(\Lambda)$, $\Delta := \Delta_2(\Lambda,\bbZ + \rmi \bbZ)$, let $\eta < \max\lbrace 1, \delta \rbrace / 4$ be arbitrary but fixed and set
        \begin{equation*}
            \calB_\eta := \bigcup_{m,n \in \bbZ} B_\eta(\lambda_{m,n}).
        \end{equation*}
        Then, according to \cite[Theorem~5 on p.~725--726]{Earl1966On}, $g_{M,N}$ is well-defined (i.e.,~$\alpha_{M,N}$ is finite) and has the following properties:
        \begin{enumerate}
            \item $g_{M,N}$ has simple zeroes at $\Lambda$ and no others;
            \item for all $\epsilon > 0$, there exists $r_0 > 0$ depending only on $\epsilon$, $\eta$, $\delta$ and $\Delta$ such that,
            \begin{enumerate}
                \item for all $z \not\in B_{r_0}(\lambda_{M,N})$, $\abs{g_{M,N}(z)} \leq \rme^{(\pi/2+\epsilon)\abs{z - \lambda_{M,N}}^2}$,
                \item for all $z \not\in B_{r_0}(\lambda_{M,N}) \cup \calB_\eta$, $\smash{\abs{g_{M,N}(z)} \geq \rme^{(\pi/2-\epsilon)\abs{z - \lambda_{M,N}}^2}}$,
                \item for all $m,n \in \bbZ$, 
                \begin{equation*}
                    \lambda_{m,n} \not\in B_{r_0}(\lambda_{M,N}) \implies \abs{g'_{M,N}(\lambda_{m,n})} \geq \rme^{(\pi/2-\epsilon)\abs{\lambda_{m,n} - \lambda_{M,N}}^2};
                \end{equation*}
            \end{enumerate} 
            \item there exists $C > 0$ depending only on $\eta$, $\delta$ and $\Delta$ such that,
            \begin{enumerate}
                \item for all $z \not\in \calB_\eta$, $\abs{g_{M,N}(z)} \geq C$, 
                \item for all $m,n \in \bbZ$, $\abs{g'_{M,N}(\lambda_{m,n})} \geq C$.
            \end{enumerate}
        \end{enumerate}

        Next, introduce the function
        \begin{equation*}
            \phi_{M,N}(z) := g_{M,N}(z) \sum_{m,n \in \bbZ} \frac{f(\lambda_{m,n})}{(z-\lambda_{m,n}) g'_{M,N}(\lambda_{m,n})}
        \end{equation*}
        and note that $\lim_{z \to \lambda_{m,n}} \phi_{M,N}(z) = f(\lambda_{m,n})$, for all $m,n \in \bbZ$. The series converges locally uniformly in $z$. Therefore, the function 
        \begin{equation*}
            \Phi_{M,N}(z) := \frac{f(z) - \phi_{M,N}(z)}{g_{M,N}(z)}
        \end{equation*}
        has removable singularities at $\lambda_{m,n}$ and extends to an entire function.

        Finally, let $\epsilon > 0$ be arbitrary but fixed and note that condition~\eqref{eq:focklikecondition} implies that there exists an increasing, diverging sequence $(r_k)_{k = 1}^\infty \in (1,\infty)$ such that
        \begin{equation*}
            \abs{f(z)} \leq \rme^{(\underline \tau+\epsilon) r_k^2}, \qquad z \in C_{r_k},
        \end{equation*}
        for $k \in \bbN$. Then, let $\gamma_k : [0,1] \to \bbC$ be an arbitrary closed curve inside the annulus $\set{z \in \bbC}{r_k - 1 \leq \abs{z} \leq r_k}$, with winding number one around the origin, which has shortest distance to any lattice point at least $\eta$.

        In the remainder of the proof, we will bound $\Phi_{M,N}$ and $f$. For the former, we apply the triangle inequality to obtain 
        \begin{equation*}
            \abs{\Phi_{M,N}} \leq \frac{\abs{f} + \abs{\phi_{M,N}}}{\abs{g_{M,N}}}.
        \end{equation*}
        The first summand is bounded just like in \cite[equations~(5.5)--(5.7) on p.~732]{Earl1966On}: % pp. 343-344 in green notebook
        \begin{equation*}
            \frac{\abs{f(z)}}{\abs{g_{M,N}(z)}} \lesssim_{\eta,\delta,\Delta} \rme^{-(\pi/2 - \underline \tau) r_k^2 / 4}, \qquad z \in \gamma_k([0,1]),
        \end{equation*}
        for $k$ sufficiently large (depending on $M$ and $N$).
        The second summand is bounded as in \cite[equations~(5.11)--(5.20) on pp.~733--734]{Earl1966On}: for all $z \not\in \calB_\eta$ and all \emph{$M,N \in \bbZ$ such that $\abs{\lambda_{M,N}}$ is sufficiently large}, % pp. 344-349 in green notebook
        \begin{equation}\label{eq:secondsummand}
            \frac{\abs{\phi_{M,N}(z)}}{\abs{g_{M,N}(z)}} \lesssim_{\epsilon,\eta,\delta,\Delta} H(\abs{\lambda_{M,N}}) \exp\left( \frac{(\pi/2-\epsilon)(\kappa+\epsilon)H(\abs{\lambda_{M,N}})}{\pi/2 - \epsilon - (\kappa+\epsilon)\mu(\abs{\lambda_{M,N}})} \right),
        \end{equation}
        where $\mu(r) := r^{-2} H(r)$ and $\epsilon$ is sufficiently small.
        
        Once $k$ is sufficiently large (still depending on $M$ and $N$), we can conclude that 
        \begin{equation*}
            \abs{\Phi_{M,N}(z)} \lesssim_{\epsilon,\eta,\delta,\Delta} H(\abs{\lambda_{M,N}}) \exp\left( \frac{(\pi/2-\epsilon)(\kappa+\epsilon)H(\abs{\lambda_{M,N}})}{\pi/2 - \epsilon - (\kappa+\epsilon)\mu(\abs{\lambda_{M,N}})} \right),
        \end{equation*}
        for $z \in \gamma_k([0,1])$. Since $\Phi_{M,N}$ is an entire function, the maximum modulus principle implies that the above holds inside of the region bounded by $\gamma_k$. Letting $k \to \infty$, we see that the above holds everywhere.

        To finally bound $f$, we can reformulate the definition of $\Phi_{M,N}$ to see 
        \begin{equation*}
            \abs{f(z)} \leq \abs{g_{M,N}(z)} \left( \abs{\Phi_{M,N}(z)} + \frac{\abs{\phi_{M,N}(z)}}{\abs{g_{M,N}(z)}} \right)
        \end{equation*}
        outside of $\calB_\eta$. Consider a closed curve $\gamma_{M,N} : [0,1] \to \bbC$ with winding number one around $\lambda_{M,N}$ staying inside of the annulus $\smash{\set{z \in \bbC}{\rho \leq \abs{z - \lambda_{M,N}} \leq \rho+1}}$, where $\rho := \max\lbrace r_0, \Delta + 1 \rbrace$, and keeping at least $\eta$ away from $\Lambda$. On $\gamma_{M,N}([0,1])$, we have
        \begin{equation*}
            \abs{g_{M,N}(z)} \leq \rme^{(\pi/2+\epsilon)\abs{z - \lambda_{M,N}}^2} \leq \rme^{(\pi/2+\epsilon)(\rho+1)^2}
        \end{equation*}
        since $\rho+1 \geq \abs{z - \lambda_{M,N}} \geq \rho \geq r_0$. This, together with our bound on $\abs{\Phi_{M,N}}$ and equation~\eqref{eq:secondsummand}, yields  
        \begin{equation*}
            \abs{f(z)} \lesssim_{\epsilon,\eta,\delta,\Delta} H(\abs{\lambda_{M,N}}) \exp\left( \frac{(\pi/2-\epsilon)(\kappa+\epsilon)H(\abs{\lambda_{M,N}})}{\pi/2 - \epsilon - (\kappa+\epsilon)\mu(\abs{\lambda_{M,N}})} \right)
        \end{equation*}
        provided that $M,N \in \bbZ$ are such that $\abs{\lambda_{M,N}}$ is sufficiently large. By the maximum modulus principle, the above extends to the region enclosed by $\gamma_{M,N}$. On this region, it follows that 
        \begin{equation*}
            \abs{f(z)} \lesssim_{\epsilon,\eta,\delta,\Delta} H(\abs{z} + \rho) \exp\left( \frac{(\pi/2-\epsilon)(\kappa+\epsilon)H(\abs{z} + \rho)}{\pi/2 - \epsilon - (\kappa+\epsilon)\mu(\abs{z} - \rho)} \right),
        \end{equation*}
        which, since the balls $B_\rho(\lambda_{M,N})$, $M,N \in \bbZ$, cover the complex plane, remains true for all $z \in \bbC$ with $\abs{z}$ sufficiently large. Using that $H(r) \to \infty$ as well as $\mu(r) \to 0$ as $r \to \infty$ and 
        \begin{equation*}
            1 \leq \frac{H(r + \rho)}{H(r)} = \frac{\mu(r + \rho)}{\mu(r)} \cdot \frac{(r + \rho)^2}{r^2} \leq \frac{(r + \rho)^2}{r^2},
        \end{equation*}
        it follows that 
        \begin{align*}
            \limsup_{r \to \infty} \frac{\log \max_{z \in C_r} \abs{f(z)}}{H(r)} &\leq \limsup_{r \to \infty} \frac{(\pi/2-\epsilon)(\kappa+\epsilon)}{\pi/2 - \epsilon - (\kappa+\epsilon)\mu(r - \rho)} \cdot \frac{H(r + \rho)}{H(r)} \\
            &= \frac{(\pi/2-\epsilon)(\kappa+\epsilon)}{\pi/2 - \epsilon}.
        \end{align*}
        Since $\epsilon > 0$ can be chosen arbitrarily small, the theorem has been proven.
    \end{proof}

    \begin{remark}
        It is plausible that a more elegant proof of the above result could be obtained using an interpolation formula due to Whittaker \cite[Theorem~1 on p.~112]{Whittaker1930On} (see also \cite[Equation~(7) on p.~68]{Polya1934Bemerkung}, \cite[Equation~(1.4) on p.~306]{Pfluger1937On} or \cite[Equation~(10.1) on p.~202]{Pfluger1945Ueber}). We chose not to pursue this route because the direct proof we provide is short and has the advantage of applying without modification to functions holomorphic outside an arbitrarily large ball.
    \end{remark}

    We mainly want to apply Theorem~\ref{thm:Earl} to entire functions $f \in \calO(\bbC)$ (satisfying condition~\eqref{eq:focklikecondition}) for which there exists an entire function $g \in \calO(\bbC)$ of exponential type along with a uniformly discrete set $\Lambda \subset \bbC$ with lower Beurling density $D^{-}(\Lambda) > 2\underline \tau/\pi$ such that $\abs{f} \lesssim \abs{g}$ on $\Lambda$. In this setting, we note
    \begin{equation*}
        \limsup_{j \to \infty} \frac{\log \abs{f(\lambda_j)}}{\abs{\lambda_j}} \leq \limsup_{j \to \infty} \frac{\log \abs{g(\lambda_j)}}{\abs{\lambda_j}} = \kappa,
    \end{equation*}
    where $\kappa \in [0,\infty)$ is the exponential type of $g$. Therefore, Theorem~\ref{thm:Earl} implies that 
    \begin{equation}
        \limsup_{r \to \infty} \frac{\log \max_{z \in C_r} \abs{f(z)}}{r} \leq \kappa
    \end{equation}
    such that $f$ is of exponential type at most $\kappa$. We summarise the above in the following corollary.

    \begin{corollary}\label{cor:EarlExponentialType}
        Let $f, g \in \calO(\bbC)$ be entire functions such that $f$ satisfies condition~\eqref{eq:focklikecondition} and $g$ is of exponential type $\kappa \in [0, \infty)$. Let $\Lambda \subset \bbC$ be a uniformly discrete set with lower Beurling density $D^{-}(\Lambda) > 2 \underline \tau / \pi$. If $\abs{f} \lesssim \abs{g}$ on $\Lambda$, $f$ is of exponential type at most $\kappa$.
    \end{corollary}

    \subsection{Main Result: Phase retrieval from structured discrete sets and shifted lattices}

    In view of Corollary~\ref{cor:EarlExponentialType}, we are now ready to prove our main result: the modulus of an entire function of exponential type, known only on a suitably dense and structured discrete set, uniquely determines it among all other entire functions satisfying condition~\ref{eq:focklikecondition} (up to a global phase).

    \begin{theorem}[Main theorem]\label{thm:MainTheorem}
        Let $f, g \in \calO(\bbC)$ be entire functions such that $f$ satisfies condition~\eqref{eq:focklikecondition} and $g$ is of exponential type $\kappa \in [0, \infty)$. Let $\Lambda \subset \bbC$ be a uniformly discrete set with lower Beurling density $D^{-}(\Lambda) > 2 \underline \tau / \pi$, and assume that $\Lambda$ contains two sequences lying on parallel rays with $2 \kappa$ times their distance strictly less than $\pi$; each satisfying the density condition~\eqref{eq:FuchsCondition} after an appropriate translation-rotation.

        More precisely, assume that there exist $\theta \in (-\pi,\pi]$ and points $z_1,z_2 \in \mathbb C$ such that $(z_1-z_2) \mathrm e^{- \mathrm i \theta} \not\in \mathbb R$ and the rays
        \[
            R_j := \set{ z_j + t e^{i\theta} }{ t > 0 }, \qquad j=1,2,
        \]
        satisfy
        \[
            2 \kappa \cdot \mathrm{dist}(R_1,R_2) = 2 \kappa \cdot \abs{\Im (\mathrm{e}^{-\mathrm{i}\theta}(z_1 - z_2))} < \pi.
        \]
        Assume further that there exist strictly increasing sequences $(t_{j,n})_{n \in\mathbb N} \subset (0,\infty)$ such that
        \[
            \lambda_{j,n} := z_j + t_{j,n} e^{i\theta} \in \Lambda, \qquad j=1,2,
        \]
        and such that the sequences $(t_{j,n})_{n\in\mathbb N}$ of positive real numbers satisfy the density condition~\eqref{eq:FuchsCondition}.
        
        Then, $f \sim g$ if and only if $\abs{f} = \abs{g}$ on $\Lambda$.
    \end{theorem}

    \begin{remark}
        \begin{enumerate}
            \item When $\kappa = 0$, the condition on the distance of the rays is automatically satisfied and equation~\ref{eq:FuchsCondition} boils down to the simpler M\"untz-Sz\'asz type condition
            \begin{equation*}
                \sum \frac{1}{\lambda_n} = \infty.
            \end{equation*}
            \item The conditions on the set $\Lambda \subset \mathbb C$ where refined during the revision thanks to a remark by one of the reviewers.
            \item The same reviewer, in fact, also pointed out that it suffices to assume that $\Lambda$ contains two sequences lying on rays, whose unique extensions to lines meet at an angle $\theta \not\in \pi \mathbb Q$; each satisfying the density condition~\eqref{eq:FuchsCondition} after an appropriate translation-rotation. Indeed, in this case, we can follow the proof below to reach a rotation symmetry in $m_f - m_g$ (instead of a translation symmetry). In particular, $m_f - m_g$ is invariant under rotation by angle $2\theta$ around some point on the real axis. Assuming by contradiction that $m_f - m_g > 0$ at some point $z \in \bbC$ iteratively leads to $m_f - m_g > 0$ on infinitely many distinct points lying on a circle. Thus, the zeroes of $f$ have an accumulation point and $f = 0 = g$ identically: a contradiction. The same argument works when the roles $f$ and $g$ are switched such that we can conclude $m_f = m_g$, and the proof follows like below. In fact, this idea is very similar to \cite[Theorem~3.3 on p.~419]{Jaming2014Uniqueness}.
        \end{enumerate}
    \end{remark}

    \begin{proof}[Proof of Theorem~\ref{thm:MainTheorem}]
        In this proof, we need the following notation: given an entire function $f \in \calO(\bbC)$, let $m_f : \bbC \to \bbN_0$ denote the function assigning the multiplicity of $z$ as a zero of $f$ to every $z \in \bbC$. Note that we use the convention $m_f(z) = 0$ if $z$ is no zero of $f$.

        It is clear that $f \sim g$ implies $\abs{f} = \abs{g}$ and so we begin by assuming that $\abs{f} = \abs{g}$ on $\Lambda$. Corollary~\ref{cor:EarlExponentialType} shows that $f$ is of exponential type less or equal than $\kappa$. We may assume, without loss of generality, that $R_1 = (0,\infty)$, $z_1 = 0$ and hence $\theta = 0$ as well as $\mathrm{dist}(R_1, R_2) = \abs{\Im z_2}$: indeed, the rigid motion $z \mapsto \rme^{\rmi\theta} z + z_1$ preserves the properties of $f$, $g$ and $\Lambda$, and maps $(0,\infty)$ to $R_1$.
        
        Consider the entire functions 
        \begin{equation*}
            f_1(z) := f(z) \overline{ f(\overline z) }, \qquad g_1(z) := g(z) \overline{ g(\overline z) }.
        \end{equation*}
        Note that, since $f$ and $g$ are of exponential type at most $\kappa$, the products $f_1$, $g_1$ have exponential type at most $2 \kappa$. Because $f_1$ and $g_1$ agree on the uniformly discrete sequence $(\lambda_{1,n})_{n\in\mathbb N}$ of positive real numbers satisfying equation~\eqref{eq:FuchsCondition}, Fuchs' theorem (Theorem~\ref{thm:Fuchs}) implies that $f_1 = g_1$ identically. Since $f_1 = g_1$, their zero sets (with multiplicities) agree. In terms of the zeroes of $f$ and $g$, this means that 
        \begin{equation*}
            m_f(z) + m_f(\overline z) = m_g(z) + m_g(\overline z), \qquad z \in \bbC.
        \end{equation*}
        The same argument applied to 
        \begin{equation*}
            f_2(z) := f(z + z_2) \overline{ f(\overline z + z_2) }, \qquad g_2(z) := g(z + z_2) \overline{ g(\overline z + z_2) },
        \end{equation*}
        yields 
        \begin{equation*}
            m_f(z + z_2) + m_f(\overline z + z_2) = m_g(z + z_2) + m_g(\overline z + z_2), \qquad z \in \bbC,
        \end{equation*}
        such that 
        \begin{equation*}
            m_f(z) - m_g(z) = m_f(z + 2 \rmi \Im z_2) - m_g(z + 2 \rmi \Im z_2), \qquad z \in \bbC.
        \end{equation*}
        This shows that the difference $m_f - m_g$ is invariant under translation by $2 \rmi \Im z_2$ such that 
        \begin{equation*}
            m_f(z) - m_g(z) = m_f(z + 2 \rmi \Im z_2 \cdot n) - m_g(z + 2 \rmi \Im z_2 \cdot n), \quad z \in \bbC,~n \in \bbN_0.
        \end{equation*}

        Now suppose, by contradiction, that there exists $z_0 \in \bbC$ such that $m_f(z_0) > m_g(z_0) \geq 0$. Then, we have $m_f(z_0 + 2 \rmi \Im z_2 \cdot n) > m_g(z_0 + 2 \rmi \Im z_2 \cdot n) \geq 0$ such that $f$ vanishes at all points $z_0 + 2\rmi \Im z_2 \cdot n$, for $n \in \bbN_0$. Define the entire function $h(z) := f(z_0 + 2 \rmi \Im z_2 \cdot z)$ which then vanishes on $\bbN_0$. Since $f$ has exponential type at most $\kappa$, $h$ is of exponential type at most $2 \kappa \cdot \abs{\Im z_2} < \pi$. By Fuchs' theorem (see also item~2 in Remark~\ref{rem:FuchsSharp} or use the weaker Carlson's theorem \cite[see, e.g., Section~5.8.1 on p.~186]{Titchmarsh1939Theory}), it follows that $h = 0$, and hence $f = 0$; a contradiction to the assumption that $m_f(z_0) > 0$. The same argument applies when we switch the roles of $f$ and $g$. So, we conclude that $m_f(z) = m_g(z)$ for all $z \in \bbC$.

        If $g = 0$ identically, then $m_f = m_g$ shows that $f = 0$ as well, and thus $f \sim g$. If $g$ is non-zero, then $m_f = m_g$ implies that $f$ is too. The Hadamard factorisation theorem (Theorem~\ref{thm:HadmardFactorisation}) shows that $f$ and $g$ factor as
        \begin{equation*}
            f(z) = z^k \rme^{P(z)} \prod_{j = 1}^\infty E\left( \frac{z}{z_j}; p \right), \qquad g(z) = z^k \rme^{Q(z)} \prod_{j = 1}^\infty E\left( \frac{z}{z_j}; p \right),
        \end{equation*}
        where $k \in \mathbb{N}_0$, $(z_j)_{j = 1}^\infty \in \mathbb{C} \setminus \{0\}$, $p \in \{0, 1\}$ and $P, Q \in \bbC[z]$ are polynomials of degree at most one. Write $P(z) = a_1 z + a_0$ and $Q(z) = b_1 z + b_0$. From the identity $f_1 = g_1$, we know that $\abs{f} = \abs{g}$ on $\bbR$ which implies $\Re P = \Re Q$ on $\bbR$. This yields $\Re a_0 = \Re b_0$ and $\Re a_1 = \Re b_1$. Similarly, from $f_2 = g_2$, we know $\abs{f} = \abs{g}$ on $\bbR + i \Im z_2$. So, $\Re P = \Re Q$ on $\bbR + \rmi \Im z_2$, which forces $\Im a_1 = \Im b_1$. Combining these, we find that $a_1 = b_1$ and $\Re a_0 = \Re b_0$. So, $P - Q = \rmi \alpha$ for some $\alpha \in \bbR$. Hence, $f = \rme^{\rmi \alpha} g$ and thus $f \sim g$, as claimed.
    \end{proof}

    The conditions on $\Lambda$ in our main theorem are perhaps most natural when viewed from the perspective of shifted lattices. Let $\omega_1, \omega_2 \in \bbC$ be nonzero complex numbers such that $\omega_1/\omega_2 \not\in \bbR$ and let $z_0 \in \bbC$. Remember that a shifted lattice $\Lambda$ is a set of the form
    \begin{equation*}
        \Lambda = z_0 + \omega_1 \bbZ + \omega_2 \bbZ = \set{ z_0 + \omega_1 m + \omega_2 n }{ m, n \in \bbZ }.
    \end{equation*}
    The vectors $\omega_1$ and $\omega_2$ span a parallelogram in $\bbC$ with area $\abs{\Im(\omega_1 \overline \omega_2)}$, which we call the \emph{covolume} of $\Lambda$, denoted $\operatorname{covol}(\Lambda)$. This is of interest because the lower Beurling density of a shifted lattice is the reciprocal of its covolume:
    \begin{equation*}
        D^{-}(\Lambda) = \frac{1}{\operatorname{covol}(\Lambda)}.
    \end{equation*}

    Additionally, a shifted lattice contains a pair of sequences lying on parallel rays, each satisfying the density condition~\eqref{eq:FuchsCondition} after an appropriate translation-rotation. Indeed, consider the rays $R_1 = z_0 + \omega_2 (0,\infty)$ and $R_2 = z_0 + \omega_1 + \omega_2 (0,\infty)$, which contain the sequences $z_0 + \omega_2 \bbN$ and $z_0 + \omega_1 + \omega_2 \bbN$, respectively, both satisfying that $\lambda_n := \lvert \omega_2 \rvert \cdot n$ is a translation-rotation of them. Therefore, $\liminf n \lambda_n^{-1} = \lvert \omega_2 \rvert^{-1}$ which implies condition~\eqref{eq:FuchsCondition} when $2 \kappa \cdot \lvert \omega_2 \rvert < \pi$. The distance between the rays is given by $\operatorname{dist}(R_1, R_2) = \abs{\Im(\omega_1 \overline{\omega_2})}/\abs{\omega_2}$, and the following statement follows immediately from our main result.

    \begin{corollary}\label{cor:ShiftedLattice}
        Let $f, g \in \calO(\bbC)$ be entire functions such that $f$ satisfies condition~\eqref{eq:focklikecondition} and $g$ is of exponential type $\kappa \in [0, \infty)$. Let $\Lambda = z_0 + \omega_1 \bbZ + \omega_2 \bbZ \subset \mathbb{C}$ be a shifted lattice with covolume $\operatorname{covol}(\Lambda) < \pi/(2\underline \tau)$ and assume
        \begin{equation*}
            2 \kappa \cdot \abs{\omega_2} < \pi, \qquad 2\kappa \cdot \frac{\abs{\Im(\omega_1 \overline{\omega_2})}}{\abs{\omega_2}} < \pi.
        \end{equation*}
        Then, $f \sim g$ if and only if $\abs{f} = \abs{g}$ on $\Lambda$.
    \end{corollary}

    \subsection{The polynomial case and applications to Gabor phase retrieval}\label{sec:Polynomials}

    Our main result, Theorem~\ref{thm:MainTheorem}, holds, in particular, when $g \in \bbC[z]$ is a polynomial. In this specific case, we can, however, derive a slightly stronger result.

    \begin{proposition}\label{prop:PolynomialProposition}
        Let $f \in \calO(\bbC)$ be an entire function satisfying condition~\eqref{eq:focklikecondition} and let $g \in \bbC[z]$ be a polynomial of degree $q \in \bbN_0$. Let $\Lambda \subset \bbC$ be a uniformly discrete set with lower Beurling density $D^{-}(\Lambda) > 2 \underline \tau / \pi$. 
        \begin{enumerate}
            \item If $\abs{f} \lesssim \abs{g}$ on $\Lambda$, then $f \in \bbC[z]$ is a polynomial of degree $q$.
            \item Moreover, $f \sim g$ if and only if $\abs{f} = \abs{g}$ on $\Lambda$.
        \end{enumerate}
    \end{proposition}

    \begin{proof}
        \begin{enumerate}
            \item Since $\abs{f} \lesssim \abs{g}$ on $\Lambda = \lbrace \lambda_j \rbrace_{j = 1}^\infty \subset \bbC$ and $g$ is a polynomial of degree $q \in \bbN_0$, we have 
            \begin{equation*}
                \limsup_{j \to \infty} \frac{\log \abs{f(\lambda_j)}}{\log(1+ \abs{\lambda_j})} \leq \limsup_{j \to \infty} \frac{\log \abs{g(\lambda_j)}}{\log(1+ \abs{\lambda_j})} = q.
            \end{equation*}
            Therefore, Theorem~\ref{thm:Earl} implies that 
            \begin{equation*}
                \limsup_{r \to \infty} \frac{\log \max_{z \in C_r} \abs{f(z)}}{\log(1+r)} \leq q.
            \end{equation*}
            In particular,
            \begin{equation*}
                \abs{f(z)} \lesssim (1 + \abs{z})^{q+1/2}, \qquad z \in \bbC.
            \end{equation*}
            We can now conclude by a standard argument\footnote{The author learned this argument through the answer of Adamski to \url{https://math.stackexchange.com/questions/143468/entire-function-bounded-by-a-polynomial-is-a-polynomial}}: $f$ is an entire function and is therefore given by the convergent power series 
            \begin{equation*}
                f(z) = \sum_{j = 0}^\infty \frac{f^{(j)}(0)}{j!} z^j.
            \end{equation*}
            Additionally, our estimate above together with Cauchy's estimate says that
            \begin{equation*}
                \abs{f^{(j)}(0)} \lesssim \frac{j!}{r^j} \cdot (1 + r)^{q+1/2}, \qquad j \in \bbN_0,~r > 0.
            \end{equation*}
            Letting $r$ go to infinity, we see that $f^{(j)}(0) = 0$ for $j > q$ and thus $f \in \bbC[z]$ is a polynomial of degree $q$ as claimed.

            \item It is clear that $f \sim g$ implies that $\abs{f} = \abs{g}$. So, assume that $\abs{f} = \abs{g}$ on $\Lambda$ and note that item~1 implies that $f$ is a polynomial of degree $q$. By a direct computation,
            \begin{equation*}
                \frac{f(z)}{g(z)} = \frac{a}{b} + \frac{h(z)}{b g(z)}
            \end{equation*}
            for $z \in \mathbb C$ not a zero of $g$, where $a,b \in \mathbb C \setminus \lbrace 0 \rbrace$ are the leading coefficients of $f$ and $g$, respectively, and $h \in \mathbb C [z]$ has degree $< q$. Taking absolute values and using the existence of a sequence $(\lambda_j)_{j = 1}^\infty \subset \Lambda$ tending to infinity, we obtain $\smash{a/b = \rme^{\rmi \alpha}}$ for some $\alpha \in \mathbb R$.

            Suppose by contradiction that $f \not\sim g$ and expand $f/g - \smash{\mathrm{e}^{\mathrm i \alpha}} = h / (bg)$ into a Laurent series
            \begin{equation*}
                \frac{f(z)}{g(z)} - \mathrm e^{\mathrm i \alpha} = \sum_{k = -\infty}^\infty c_k z^{-k}
            \end{equation*}
            outside of a sufficiently large ball. Let $\gamma_r : [0,2\pi] \to \mathbb C$ denote the contour $\gamma_r(t) = r \mathrm e^{\mathrm i t}$ around $B_r$. Since the right-hand side of
            \begin{equation*}
                c_k = \frac{1}{2\pi \mathrm i b} \oint_{\gamma_r} \frac{z^{k-1} h(z)}{g(z)} \dd z = \frac{1}{2\pi b} \int_0^{2\pi} \frac{(r \mathrm e^{\mathrm i t})^k h(r \mathrm e^{\mathrm i t})}{g(r \mathrm e^{\mathrm i t})} \dd t
            \end{equation*}
            tends to zero as $r \to \infty$ when $k \leq 0$, we conclude $c_k = 0$ for all $k \leq 0$, and hence
            \begin{equation*}
                \frac{f(z)}{g(z)} = \mathrm e^{\mathrm i \alpha} + \sum_{k = k_0}^\infty c_k z^{-k},
            \end{equation*}
            for some $k_0 > 0$ with $c_{k_0} \neq 0$ (where we used that $f \not\sim g$ to guarantee that some such non-zero coefficient exists).

            It is immediate that
            \begin{align*}
                \frac{\lvert f(z) \rvert^2}{\lvert g(z) \rvert^2} &= 1 + 2 \Re \left[\mathrm e^{-\mathrm i \alpha} \sum_{k = k_0}^\infty c_k z^{-k}\right] + \left\lvert \sum_{k = k_0}^\infty c_k z^{-k} \right\rvert^2 \\
                &\geq 1 + 2 \Re \left[\mathrm e^{-\mathrm i \alpha} c_{k_0} z^{-k_0}\right] - 2 \sum_{k = k_0+1}^\infty \lvert c_k \rvert \lvert z \rvert^{-k}.
            \end{align*}
            Since $D^-(\Lambda) > 0$, we may find arbitrarily large $\lambda \in \Lambda$ within angle $\theta$ of $\smash{w:=(\mathrm e^{\mathrm i \alpha} \overline c_{k_0})^{-1/k_0}}$, for any fixed $0 \leq \theta < \pi / (2k_0)$. For such $\lambda$,
            \begin{align*}
                1 = \frac{\lvert f(\lambda) \rvert^2}{\lvert g(\lambda) \rvert^2} &\geq 1 + 2 \Re \left[(z \overline w)^{-k_0}\right] - 2 \sum_{k = k_0+1}^\infty \lvert c_k \rvert \lvert \lambda \rvert^{-k} \\
                &\geq 1 + 2 \lvert \lambda \rvert^{-k_0} \lvert c_{k_0} \rvert \cos (k_0 \theta) - 2 \sum_{k = k_0+1}^\infty \lvert c_k \rvert \lvert \lambda \rvert^{-k}.
            \end{align*}
            Once $\lvert \lambda \rvert$ is sufficiently large, the second term dominates the third, yielding $1 > 1$, a contradiction. We conclude $f \sim g$.
        \end{enumerate}

        \vspace{-12pt}
    \end{proof}

    \begin{remark}
        The second part of the above proof was suggested by one of the reviewers.
    \end{remark}

    As noted in the introduction, all of our results carry over to the setting of Gabor phase retrieval. In the case of Proposition~\ref{prop:PolynomialProposition}, the corresponding implication is particularly interesting and we, therefore, highlight it in the following.
    
    Let $f \in L^2(\bbR)$ and suppose that there exists a linear combination of Hermite functions $\smash{g = \sum_{j = 1}^q \lambda_j h_j}$---where we remind the reader that $h_j$ denotes the $j$-th Hermite function for $j \in \bbN_0$---such that $\smash{\abs{\calG f} = \abs{\calG g}}$ on a uniformly discrete set $\Lambda$ with lower Beurling density $\smash{D^{-}(\Lambda) > 1}$ containing a pair of infinite subsets lying on two parallel lines. Then, after taking the Bargmann transform, we have that $\abs{\calB f} = \abs{\calB g}$ on $\smash{\overline \Lambda := \set{ \overline \lambda }{ \lambda \in \Lambda }}$, where $\smash{\calB g = \sum_{j = 1}^q \lambda_j e_j}$ is a polynomial of degree $q$---with the notation $e_j(z) = (\pi^j/j!)^{1/2} z^j$ from the introduction. Note that $\calB f$ is an entire function and that 
    \begin{equation*}
        \underline \tau = \liminf_{r \to \infty} \frac{\log \max_{z \in C_r} \abs{\calB f(z)}}{r^2} \leq \frac{\pi}{2}
    \end{equation*}
    according to equation~\eqref{eq:bound_type_Bargmann}. Therefore, Proposition~\ref{prop:PolynomialProposition}, item~2, implies that $\calB f \sim \calB g$ and taking the inverse Bargmann transform yields $f \sim g$ as desired. We can conclude that finite linear combinations of Hermite functions can be uniquely recovered (up to a global phase) from their Gabor transform magnitudes sampled on suitably dense and structured discrete sets. Since the Hermite functions form an orthonormal basis of $L^2(\bbR)$, their linear span is dense in $L^2(\bbR)$ and we can conclude that there is a dense set of functions in $L^2(\bbR)$ which are uniquely recoverable.

    \begin{corollary}\label{thm:GaborCorollary}
        Let $f \in L^2(\bbR)$ and let $g = \sum \lambda_n h_n$ be a linear combination of Hermite functions. Let $\Lambda \subset \bbR^2$ be a uniformly discrete set with lower Beurling density $D^{-}(\Lambda) > 1$. Then, $f \sim g$ if and only if $\abs{\calG f} = \abs{\calG g}$ on $\Lambda$.
        
        Therefore, the set of functions $f \in L^2(\bbR)$ which can be recovered uniquely (up to a global phase) from their Gabor transform magnitudes sampled on $\Lambda$ is dense in $L^2(\bbR)$.
    \end{corollary}

    \begin{remark}
        The results in this section hold, in particular, for shifted lattices with covolume strictly smaller than $\pi/(2\underline \tau)$ and $1$, respectively.
    \end{remark}

    \subsection{On the counterexamples and the necessity of certain constraints in our main results}\label{sec:Counterexamples}

    As discussed in the introduction, counterexamples for the lines $a \bbZ \times \bbR$ already appeared in \cite[cf.~Remark~3 on p.~6]{Alaifari2022Phase}. These examples have Gabor transforms of the form
    \begin{equation*}
        \rme^{-\pi\rmi x\omega} \rme^{-\frac{\pi}{2}\left(x^2 + \omega^2\right)} \left( \cosh\left( \frac{\pi}{2a} \left( \omega + \rmi x \right) \right) \mp \rmi \sinh\left( \frac{\pi}{2a} \left( \omega + \rmi x \right) \right) \right), \qquad x,\omega \in \bbR,
    \end{equation*}
    up to a real multiplicative constant. After applying the Bargmann transform, one obtains (again up to a constant) the entire functions
    \begin{equation*}
        \cosh\left( \frac{\pi \rmi}{2a} z \right) \mp \rmi \sinh\left( \frac{\pi \rmi}{2a} z \right), \qquad z \in \bbC,
    \end{equation*}
    which can be written more compactly as
    \begin{equation*}
        \cos\left( \frac{\pi z}{2a} \right) \pm \sin\left( \frac{\pi z}{2a} \right) \eqsim \sin\left( \frac{\pi}{4} \pm \frac{\pi z}{2a} \right), \qquad z \in \bbC.
    \end{equation*}
    These are entire functions of exponential type $\pi / (2a)$ whose moduli agree on $a \bbZ + \rmi \bbR$ but which differ by more than a global phase.

    \begin{figure}
        \centering
        \begin{tikzpicture}

            % Coordinates
            \coordinate (O) at (0,0);
            \coordinate (w1) at (4,4/3);
            \coordinate (w2) at (0,2.5);
            \coordinate (e) at (14/15,14/45);

            % Coordinate system
            \draw[help lines,->] (-2,0) -- (6,0);
            \draw[help lines,->] (0,-.5) -- (0,4);

            % Lattice
            \draw[help lines] (4,-.5) -- (4,4);
            \draw[help lines] (-1.5,-.5) -- (6,2);
            \draw[help lines] (-1.5,2) -- (4.5,4);

            % Distances between lines
            \draw[help lines,<->] (O) -- (4,0);
            \node[help lines,below] at (2,0) {$\Re \omega_1$};
            \draw[help lines,<->] (w2) -- (e);
            \node[help lines,right] at (14/30,253/180) {$d$};

            % Angles
            \pic[draw,angle radius=12,angle eccentricity=0.5] {angle=w1--O--w2};
            \pic[draw,angle radius=10,angle eccentricity=0.5,help lines] {angle=w2--e--O};
            \node at (0.13,0.2) {$\theta$};
            \node[help lines] at (0.75,0.39) {$\cdot$};

            % Generating vectors
            \draw[->] (O) -- (w1);
            \node[below right] at (2,4/6) {$\omega_1$};
            \draw[->] (O) -- (w2);
            \node[left] at (0,1.25) {$\omega_2$};

        \end{tikzpicture}
        
        \caption{There are two families of parallel lines containing the lattice $\Lambda$: (i)~$\omega_1 \mathbb Z + \omega_2 \mathbb R = \Re \omega_1 \cdot \mathbb Z + \mathrm i \mathbb R$, and (ii)~$\omega_1 \mathbb R + \omega_2 \mathbb Z = \mathrm e^{-\mathrm i \theta}(d \mathbb Z + \mathrm i \mathbb R)$.}
        
        \label{fig:parallel_lines_lattice}
    \end{figure}

    The counterexamples demonstrate that our recurring assumption---namely, that $\Lambda$ contains a pair of sequences lying on two parallel rays with $2 \kappa$ times their distance strictly less than $\pi$, each satisfying the density condition~\eqref{eq:FuchsCondition} after an appropriate translation-rotation---is tight when $\Lambda$ is a rectangular shifted lattice. To make this precise, consider a general shifted lattice $\Lambda = z_0 + \omega_1 \bbZ + \omega_2 \bbZ$ and assume that $\kappa > 0$ because the recurring assumption is vacuously satisfied when $\kappa = 0$. Without loss of generality (by applying a rigid motion and exchanging $\omega_1$ and $\omega_2$ if necessary), we may assume that $z_0 = 0$, $\omega_2 = \mathrm i \cdot \lvert \omega_2 \rvert$, and $\omega_1 \in \mathbb{C}$ lies in the right half-plane. There are two natural sets of infinitely many parallel lines of equal distance which contain this shifted lattice. The slightly simpler of the two is $\omega_1 \mathbb Z + \omega_2 \mathbb R = \Re \omega_1 \cdot \mathbb Z + \mathrm i \mathbb R \supset \Lambda$. From this, we see that $2 \kappa \cdot \Re \omega_1 \geq \pi$ implies the existence of two entire functions of exponential type at most $\kappa$ whose moduli agree on $\Lambda$ but which differ by more than a global phase. Specifically, those functions are 
    \begin{equation*}
        \sin\left( \frac{\pi}{4} \pm \frac{\pi z}{2 \Re \omega_1} \right).
    \end{equation*}
    The other set of infinitely many parallel lines is $\omega_1 \mathbb R + \omega_2 \mathbb Z = \mathrm e^{-\mathrm i \theta}(d \mathbb Z + \mathrm i \mathbb R) \supset \Lambda$, where $\theta$ is the angle between $\omega_1$ and $\omega_2$ and $d = \lvert \omega_2 \rvert \sin \theta = \lvert \omega_2 \rvert \cdot \Re \omega_1 / \lvert \omega_1 \rvert$ is the distance between $\omega_1 \mathbb R$ and $\omega_1 \mathbb R + \omega_2$ (see Figure~\ref{fig:parallel_lines_lattice}). So, $2 \kappa \cdot d \geq \pi$ implies the existence of two entire functions of exponential type at most $\kappa$ whose moduli agree on $\Lambda$ but which differ by more than a global phase, and those functions are 
    \begin{equation*}
        \sin\left( \frac{\pi}{4} \pm \frac{\pi z \mathrm e^{\mathrm i \theta}}{2 d} \right).
    \end{equation*}

    Thus, both conditions $2 \kappa \cdot \Re \omega_1 < \pi$ and $2 \kappa \cdot d < \pi$ are necessary for uniqueness up to a global phase, and these conditions correspond to 
    \begin{equation*}
        2 \kappa \cdot \frac{\abs{\Im(\omega_1 \overline{\omega_2})}}{\min\lbrace \abs{\omega_1}, \abs{\omega_2} \rbrace} < \pi.
    \end{equation*}
    We have therefore shown the following proposition.

    \begin{proposition}
        Let $\kappa > 0$ and suppose that $\Lambda = z_0 + \omega_1 \bbZ + \omega_2 \bbZ \subset \mathbb{C}$ is a shifted lattice such that
        \begin{equation*}
            2 \kappa \cdot \frac{\abs{\Im(\omega_1 \overline{\omega_2})}}{\min\lbrace \abs{\omega_1}, \abs{\omega_2} \rbrace} \geq \pi.
        \end{equation*}
        Then, there exist functions $f, g \in \calO(\bbC)$ of exponential type at most $\kappa$ such that $f \not\sim g$ and $\abs{f} = \abs{g}$ on $\Lambda$.
    \end{proposition}

    \begin{remark}
        \begin{enumerate}
            \item When $\Lambda$ is a rectangular shifted lattice---i.e., when $\Re(\omega_1 \overline \omega_2) = 0$---with sufficiently small covolume, then the proposition above together with Corollary~\ref{cor:ShiftedLattice} shows that 
            \begin{equation*}
                2 \kappa \cdot \frac{\abs{\Im(\omega_1 \overline{\omega_2})}}{\min\lbrace \abs{\omega_1}, \abs{\omega_2} \rbrace} = 2 \kappa \cdot \max\lbrace \abs{\omega_1}, \abs{\omega_2} \rbrace < \pi
            \end{equation*}
            is a necessary and sufficient condition for unique phase retrieval of entire functions of exponential type at most $\kappa$ from measurements on $\Lambda$.
            \item One might hope to replace these geometric conditions with a stronger constraint on the density or covolume of $\Lambda$. However, this is not possible: for any $\kappa > 0$ and any $\epsilon > 0$, there exists a shifted lattice $\Lambda$ of covolume $\operatorname{covol}(\Lambda) < \epsilon$ that violates the necessary condition above. For instance, taking $\omega_2 = \pi/(2\kappa)$ and $\omega_1 = \rmi \cdot \kappa\epsilon/\pi$ yields such a lattice. Hence, no smallness assumption on the covolume or largeness assumption on the lower Beurling density alone, regardless of strength, can ensure the conclusion of Corollary~\ref{cor:ShiftedLattice} or Theorem~\ref{thm:MainTheorem}.
        \end{enumerate}
    \end{remark}

    \paragraph{Acknowledgements} I would like to thank the organisers and all participants of the conference on phase retrieval in mathematics and applications (PRiMA), which took place from the $5^\mathrm{th}$ until the $9^\mathrm{th}$ of August $2024$, in Leiden; during this conference one of the questions answered in this paper was born. I would also like to thank the organisers and all participants of the FIM Workshop on phase retrieval \& Banach lattices, which took place from the $5^\mathrm{th}$ until the $9^\mathrm{th}$ of May $2025$, in Z{\"u}rich, for their inputs. In particular, my thanks goes to P.~Jaming whose suggestions are greatly appreciated. Finally, I would like to thank my wife, Alice Paula, for her support; always and, especially, during the last year. 

    \paragraph{Declaration of generative AI and AI-assisted technologies in the writing process} During the preparation of this work the author used ChatGPT and Claude in order to improve the clarity of the paper. After using these tools, the author reviewed and edited the content as needed and takes full responsibility for the content of the published article.

    \paragraph{Funding} This research did not receive any specific grant from funding agencies in the public, commercial, or not-for-profit sectors.
 
    \bibliography{lit}
    \bibliographystyle{plain}

\end{document}